\documentclass[12pt]{article} 

\usepackage{amsfonts}
\usepackage{amssymb,amsmath}
\usepackage{amsthm}

\usepackage{todonotes}

\usepackage[upright]{fourier}\usepackage{baskervald} 
\usepackage{dsfont} 

\usepackage{graphicx}
\usepackage{float}
\usepackage{wrapfig}
\usepackage{subfigure}

\usepackage{url}

\usepackage{hyperref}
\usepackage[left=30mm, right=30mm]{geometry}

\usepackage{verbatim}

\usepackage[english]{babel}
\usepackage[latin1]{inputenc}

\usepackage{calligra}
\usepackage[T1]{fontenc}

\usepackage{bbm}

\usepackage{mathrsfs}

\usepackage{color}
\usepackage[color,all]{xy}

 \usepackage{relsize}

                 \hypersetup{ pdfborder={0 0 0}, 
                              colorlinks=true, 
                              linktoc=page,
                              pdfauthor={Christian Bopp and Michael Hoff}, 
                              pdftitle={Moduli of lattice polarized K3 surfaces via relative canonical resolutions}} 

\theoremstyle{definition}
\newtheorem{definition}{Definition}[section]

\theoremstyle{plain}
\newtheorem{theorem}[definition]{Theorem}

\newtheorem*{mtheorem}{Theorem}
\newtheorem*{mcorollary}{Corollary}

\newtheorem{corollary}[definition]{Corollary}

\newtheorem{lemma}[definition]{Lemma}

\newtheorem{proposition}[definition]{Proposition}

\theoremstyle{remark}

\newtheorem{remark}[definition]{Remark}

\newcommand{\OO}{\mathscr{O}}

\newcommand{\PP}{\mathbb{P}}
\newcommand{\mscr}[1]{\mathscr{#1}}
\newcommand{\OsC}{\mathscr{O}_S(C)}
\newcommand{\OsH}{\mathscr{O}_S(H)}
\newcommand{\OsN}{\mathscr{O}_S(N)}

\newcommand{\Sing}{\text{\upshape{Sing}}}

\newcommand{\im}{\text{\upshape{Im}}}

\newcommand{\Pic}{\text{\upshape{Pic}}}

\newcommand{\cE}{\mathscr{E}}

\newcommand{\PPE}{{\PP(\cE)}}

\newcommand{\cI}{\mathscr{I}}

\newcommand{\cH}{\mathscr{H}}
\newcommand{\cF}{\mathscr{F}}

\newcommand{\cM}{\mathcal{M}}

\newcommand{\cP}{\mathscr{P}}
\newcommand{\cHom}{\mathscr{H}\!om}
\newcommand{\cEnd}{\mathscr{E}\!nd}
\newcommand{\rk}{\text{\upshape{rk}}}

\newcommand{\cW}{\mathcal{W}}
\newcommand{\hh}{\mathfrak{h}}
\newcommand{\nn}{\mathfrak{n}}

\makeatletter
\newcommand{\subjclass}[2][2010]{%
  \let\@oldtitle\@title%
  \gdef\@title{\@oldtitle\footnotetext{#1 \emph{Mathematics subject classification:} #2}}%
}
\newcommand{\keywords}[1]{%
  \let\@@oldtitle\@title%
  \gdef\@title{\@@oldtitle\footnotetext{\emph{Key words and phrases:} #1.}}%
}
\makeatother

\newcommand{\Michael}[1]{}
\newcommand{\Christian}[1]{}

\AtEndDocument{\bigskip{
  \textsc{Universit\"at des Saarlandes, Campus E2 4, D-66123 Saarbr\"ucken, Germany} \par  
  \textit{E-mail address}, C.~Bopp: \texttt{bopp@math.uni-sb.de} \par
  \textit{E-mail address}, M.~Hoff: \texttt{hahn@math.uni-sb.de} \par
}}

\begin{document}

\title{Moduli of Lattice Polarized K3 Surfaces via Relative Canonical Resolutions}
\author{Christian Bopp and Michael Hoff}
\date{\today}


\maketitle 

\begin{abstract}
For a smooth canonically embedded curve $C$ of genus $9$ together with a pencil $|L|$ of degree $6$, we study the relative canonical resolution of $C\subset X\subset \Bbb P^8$, where $X$ is the scroll swept out by the pencil $|L|$. 
We show that the second syzygy bundle in this resolution of $C\subset X$ is unbalanced.
The proof reveals a new geometric connection between the universal Brill--Noether variety $\mathcal{W}^1_{9,6}$ and a moduli space  $\mathcal{F}^\hh$ of lattice polarized $K3$ surfaces (for a certain rank $3$ lattice $\mathfrak{h}$).
As a by-product we prove the unirationality of $\mathcal{F}^\mathfrak{h}$ and show that $\mathcal{W}^1_{9,6}$ is birational to a projective bundle over a moduli space of lattice polarized $K3$ surfaces $\mathcal{F}^{\mathfrak{h}'}$ for a certain rank $4$ lattice $\mathfrak{h}'$ which contains $\mathfrak{h}$ as a sublattice.
\end{abstract}


\section{Introduction}

We consider a canonically embedded curve $C\subset \PP^{g-1}$ of genus $g$ together with base point free complete pencil $|L|=g^1_k$ of degree $k$. The linear span of divisors in the pencil $|L|$ sweeps out a rational normal scroll $X=\bigcup_{D\in |L|}\overline{D} \subset \PP^{g-1}$ of dimension $k-1$ and degree $g-k+1$. To the scroll one naturally associates a projective bundle $\pi:\PPE\to \PP^1$. Our main object of study is the minimal free resolution of $C\subset \PPE$ in terms of $\OO_\PPE$-modules.
This resolution was introduced in \cite{Sch} (see also Theorem \ref{relCanRes}) and we refer to it as the \textit{relative canonical resolution} (RCR).
 The bundles $F_i$ appearing in the RCR are isomorphic to bundles $N_i$ over $\PP^1$ (up to twists). The degree and rank of the bundles $N_i$ are known, but the splitting types for general pairs $(C,L)$ are only known in a very few cases (see \cite{B}, \cite{DP}). Naively, one expects that the splitting type is as balanced as possible (that is, the first cohomology group of the endomorphism bundle vanishes). For a more detailed conjecture of the behaviour of the splitting type see \cite{BH, BH21}. 

Knowing the generic splitting type of the bundles appearing in the RCR one can study the subloci of the Hurwitz space $\cH_{g,k}$ for which the splitting type differs from the generic one (see \cite{Pat}). On the other hand one can use the RCR to construct the canonical resolution of $C\subset \PP^{g-1}$ (see \cite{Sch} and \cite{Sag}).

In order to show that the general pair $(C,L)$ has a balanced RCR it is sufficient (by semi-continuity) to prove the existence of one example. This can for instance be done by degeneration techniques (see \cite{BP}) or computer algebra (see \cite{BH, BH21}).

In this article we treat the case of genus $9$ curves together with a pencil of degree $6$. 
Computer algebra experiments indicate that such pairs have an unbalanced second syzygy bundle $N_2$ (see Proposition \ref{existenceExample} for the precise structure of the RCR).
In this case the usual semi-continuity argument does not show that second syzygy bundle $N_2$ is  generically unbalanced. We will show the following.
\begin{mtheorem}[see Corollary \ref{allUnbalanced}]
 For any $(C,L)\in \cW^1_{9,6}$ the relative canonical resolution has an unbalanced second syzygy bundle.
\end{mtheorem}

The above theorem follows from a new geometric connection  between the universal Brill--Noether variety $\cW^1_{9,6}$ and a moduli space $\cF^\hh$ parametrizing lattice polarized $K3$ surfaces for a certain rank 3 lattice $\hh$. 

We will explain this connection and our main results. 
In our experiments the generators of $C\subset \PPE$ have linear syzygies which force the second syzygy bundle in the RCR to be unbalanced. Furthermore, there exists a linear syzygy such that the associated syzygy scheme defines a $K3$ surface $S$ of genus $8$. The curve $C$ is contained in $S$ which in turn is contained in the scroll $\PPE$.
The ruling on $\PPE$ cuts out a pencil of elliptic curves $N$ of degree $5$ on $S$ and the ruling restricted to $C$ gives back the pencil $|L|$. If we denote by $H$ the hyperplane section of $S$, then the intersection matrix defined by the classes $\big\{\OsH,\OsC,\OsN\big\}$ has the form
$$ 
\begin{pmatrix}14 & 16 & 5 \\ 16 & 16 & 6 \\ 5 & 6 & 0 \end{pmatrix}
.
$$
In order to state the main theorem we need some notation. Let $\hh$ be the rank $3$ lattice defined by the above intersection matrix with respect to an ordered basis $\{h_1,h_2,h_3\}$. The moduli space $\cF^\hh$ of lattice polarized $K3$ surfaces parametrizes pairs $(S,\varphi)$ consisting of a $K3$ surface and a primitive lattice embedding $\varphi: \hh \to \Pic(S)$ such that  $\varphi(\hh)$ contains an ample class. It is a quasi projective irreducible $17$-dimensional variety by \cite{Do96}. If $(S,\varphi)\in \cF^\hh$ is an $\hh$-polarized $K3$ surface, then we denote 
$$
\OO_S(H)=\varphi(h_1),\ \OO_S(C)=\varphi(h_2)\ \text{and } \varphi(h_3)=\OO_S(N).
$$
We consider the open subset
$$
\cF^\hh_8=\left\{ (S,\varphi)\ \big| \ S\in \cF^\hh \text{ and } \OO_S(H) \text{ ample } \right\}
$$
 of the moduli space $\cF^\hh$ and the open subset 
$$
\cP^\hh_8=\left\{(S,\varphi,C)\ \big |\ (S,\varphi)\in \cF_8^\hh \text{ and } C\in |\OsC| \text{ smooth} \right\}
$$
of the tautological $\PP^9$-bundle over $\cF^\hh_8$.
Then, the natural restriction map	
$$
\phi:\cP^\hh_8 \to \cW^1_{9,6},\ \big(S,\varphi,C\big) \mapsto \big(C,\OsN\otimes \OO_C\big)
$$
connects $\cP^\hh_8$ with the universal Brill--Noether variety
$
\cW^1_{9,6}.
$
Note that $\dim \cP^\hh_8= \dim \cW^1_{9,6}+1=26$. Our main theorem is the following.
\begin{mtheorem}[see Theorem \ref{mainthm}]
	The map $\phi:\cP^\hh_8 \to \cW^1_{9,6}$ defined above
	is dominant. 
	Moreover, $\cP^\hh_8$ is birational to a $\PP^1$-bundle over an open subset of $\cW^1_{9,6}$. 
\end{mtheorem}

For $K3$ surfaces the best studied moduli spaces are those parametrizing polarized $K3$ surfaces. In a series of papers (\cite{Muk88}, \cite{Mukg11}, \cite{Mukg13}, \cite{Mukg16}, \cite{Mukg1820}) Mukai showed that the moduli space $\cF^H_g$ parametrizing $H$-polarized $K3$ surfaces, where $H$ is an ample class with $H^2=2g-2$, is unirational for  $g\le 13$ and $g=16,18,20$. In \cite{Mukg17}, he also announced the case $g=17$. The unirationality was recently shown for  $g=14$ and $g=22$ in \cite{Nu16, FV18} and \cite{FV20}, respectively. 

On the other hand Gritsenko--Hulek--Sankaran \cite{GHS} showed that $\cF^H_g$ is of general type for $g=47,51,55,58,61$ and $g>62$ (see also \cite{Kon93} and \cite{Kon99}).

For $\nn$-polarized $K3$ surfaces where $\nn$ is a lattice of higher rank much less is known. 
For certain higher rank lattices $\nn$ the (uni)-rationality of $\cF^\nn$ was proved by Bhargava--Ho--Kumar in \cite{BHK16}, by Ma in \cite{Ma15, Ma19} and by the second author and Knutsen in \cite{HK20}. For the Nikulin lattice $\mathfrak{N}$ Farkas--Verra and Verra showed the unirationality of $\cF^\mathfrak{N}_g$ for $g\leq 8$ (see \cite{FV12, FV16} and \cite{V15}). Here $g$ refers to the self intersection $H^2=2g-2$ of the ample class $H$ in $\mathfrak{N}$. The (uni)-rationality of moduli spaces of ``non-standard'' Nikulin surfaces are studied in \cite{KLCV}. The (uni)-rationality of the moduli space of elliptic $K3$ surfaces was shown in \cite{Mir, Lej} and the moduli spaces of certain elliptic $K3$ surfaces of Picard rank $3$ were recently studied in \cite{FM20, FHM20}. 

Using the unirationality of  $\cW^1_{9,6}$, which is classically known due to \cite{AC81}, we obtain the following corollary from our main theorem.

\begin{mcorollary}[see Corollary \ref{unirationalityTheorem}]
 The moduli spaces $\cP_8^\hh$ and $\cF^\hh$ are unirational.
\end{mcorollary}

\begin{remark}
 Computer experiments indicate that a similar behaviour occurs for pairs $(C,L)\in \cW_{g,g-3}^1$. In these cases the corresponding syzygy scheme yields a $(g-7)$-dimensional Fano variety whose general linear section with a $\PP^7$ is a smooth canonically embedded curve of genus $8$. 
\end{remark}

Mukai's work \cite{Muk88} shows that the moduli space of genus $9$ curves $\cM_9$ is dominated by a projective bundle over the moduli space of polarized $K3$ surfaces $\cF^H_9$. He describes (Brill--Noether general) $K3$ surfaces containing a general curve $C\in \cM_9$. 
In contrast, we will show the existence of a unique $K3$ surface of Picard rank $4$ containing $C$ for a general point $(C,\omega_C\otimes L^{-1})\in \cW^3_{9,10}$. Our result is similar to the main results of \cite{AK11, Mukg11}. The idea is the following: Let $(C,L)\in \cW^1_{9,6}$ be a general point. Then the image $C'$ under the residual embedding $\omega_C\otimes L^{-1}$ lies on a net of quartics. We will show that the fiber $\phi^{-1}(C,L)$ of the map $\phi:\cP^\hh_8 \to \cW^1_{9,6}$ defines a plane cubic inside this net of quartics. By studying the geometry of the quartic corresponding to the singular point, it follows that its Picard lattice with respect to an ordered basis 
 $\{h_1',h_2',h_3',h_4'\}$ 
 has the form
$$
\hh'\sim 
\begin{pmatrix}
4 & 10 & 1 & 1 \\
10 & 16 & 0 & 0 \\
1 & 0 & -2 & 0 \\
1 & 0 & 0 & -2
\end{pmatrix} .
$$
This yields the following theorem. 
\begin{mtheorem}[see Theorem \ref{birationalityTheorem}] Let
$$
\cP^{\hh'}_3=\left\{(S,\varphi,C)\ \big |\ (S,\varphi)\in \cF^{\hh'}, \OO_S(H')=\varphi(h_1') \text{ ample and } C\in |\varphi(h_2')| \text{ smooth} \right\}
$$
be the open subset of the tautological $\PP^9$-bundle over the moduli space $\cF^{\hh'}_3$. Then
	the morphism 
	$$
	\phi': \cP^{\hh'}_3 \to \cW^3_{9,10},\ (S,\varphi,C) \mapsto (C, \OO_S(H')\otimes \OO_C)
	$$
	defines a birational equivalence.
\end{mtheorem}

In Section \ref{preliminaries} we recall the definition and basic results of relative canonical resolutions and lattice polarized $K3$ surfaces. 
Section \ref{unirationality} is devoted to the proof of our main theorem. 
In Section \ref{birationalityBNloci} we deduce the birational equivalence between $\cP^{\hh'}_3$ and $\cW^3_{9,10}$.

\paragraph*{Acknowledgement} 
We would like to thank Michael Kemeny, Andreas Leopold Knutsen, Kristian Ranestad and Frank-Olaf Schreyer for helpful and enlightening discussions and remarks. Andreas Leopold Knutsen and Frank-Olaf Schreyer also provided valuable feedback and correction on an early draft of the paper. 

\section{Preliminaries}\label{preliminaries}
In this section we briefly summarize the construction of relative canonical resolutions and recall some well-known facts about the moduli space of lattice polarized $K3$ surfaces.
\subsection{Relative canonical resolutions}\label{relCanRes}
Let $C\subset \PP^{g-1}$ be a canonically embedded curve of genus $g$ which admits a complete basepoint free $g^1_k$. Then 
$$
X=\bigcup_{D\in g^1_k}\overline{D_\lambda}\subset \PP^{g-1}
$$
is a rational normal scroll of dimension $\dim(X)=(k-1)$ and degree $\deg(X)=(g-k+1)$, where $\overline{D}$ denotes the span of the divisor $D$. 
The variety $X$ is the image of a projective bundle $\pi :\PPE\to \PP^1$ under the natural map
$$
j: \PP(\cE)\to \PP(H^0(\PPE,\OO_\PPE(1)))
$$
where  $\cE=\OO_{\PP^1}(e_1)\oplus\dots \oplus \OO_{\PP^1}(e_{k-1})$ with $\sum_{i=1}^{k-1} e_i=\deg(X)$.
The map $j$ above is an isomorphism if all $e_i\geq1$, otherwise it is a resolution of (rational) singularities, and we will consider $\PPE$ instead of $X$ most of the time.

Recall that the Picard group of $\PPE$ is generated by the hyperplane class $H$ and the ruling $R=[\pi^*\OO_{\PP^1}(1)]$ with intersection products
$
H^{k-1}=f,\ H^{k-2}R=1,\ R^2=0.
$
By \cite{Sch} one can resolve the canonical curve $C\subset \PPE$ in terms of $\OO_\PPE$-modules.
\begin{theorem}[\cite{Sch}, Corollary 4.4]\label{Sch4.4}
Let $C$ be a curve with a complete base point free $g^1_k$ and 
let $\PPE$ be the projective bundle associated to the scroll $X$, swept out by the $g^1_k$.
\begin{enumerate}
\item [(a)]
$C\subset \PPE$ has a resolution $F_\bullet$ of type 
$$
0 \to \pi^*N_{k-2}(-kH) \to \pi^*N_{k-3}\big((-k+2)H\big) \to \dots \to \pi^*N_1(-2H) \to \OO_\PPE \to \OO_C \to 0
$$
with $N_i=\sum_{j=1}^{\beta_i}\OO_\PPE(a_j^{(i)})$ of rank $\beta_i=\frac{k}{i+1}(k-2-i)\binom{k-2}{i-1}$. 
\item [(b)] The complex $F_\bullet$ is self dual, that is, 
$\cHom(F_\bullet, \OO_\PPE(-kH+(g-k-1)R))\cong F_\bullet$
\end{enumerate}
\end{theorem}
The resolution $F_\bullet$ is called the \emph{relative canonical resolution} of $C$ with respect to the $g^1_k$ on $C$.
By \cite{BH} the slopes of the $i^{\text{th}}$ syzygy bundle 
$
N_i
$
is known to be $\mu(N_i)=\frac{(g-k-1)(i+1)}{k}$, but the generic splitting type of $N_i$ is only known in very few cases. 

We call a bundle $N=\OO_{\PP^1}(a_1)\oplus\dots \oplus \OO_{\PP^1}(a_d)$ on $\PP^1$ \emph{balanced} if $\max_{i,j}|a_i-a_j|\leq 1$ (or equivalently $h^1(\PP,\cEnd(\cE))=0$). For example, the bundle $\cE$ defining a scroll swept out by a pencil on a Petri-general canonical curve $C$ is always balanced by \cite[(2.5)]{Sch}. We summarize what is known about the bundles in the relative canonical resolution to be balanced.

The relative canonical resolution $F_\bullet$ is generically balanced if $k\leq 5$ (see \cite{B} and \cite{DP}) or if $g=nk+1$ for some $n>1$ (see \cite{BP}). Furthermore the first bundle $N_1$ is known to be generically  balanced for $g\geq (k-1)(k-3)$ (see \cite{BP}) or if the Brill--Noether number $\rho=\rho(g,k,1)$ is non-negative and $(k-\rho-\frac{7}{2})^2-2k+\frac{23}{4}\leq 0 $ (see \cite{BH}).

In the next section we will show that the second syzygy module of the relative canonical resolution of a general point $(C,L)\in \cW^1_{9,6}$ is unbalanced and corresponds to the existence of $K3$ surfaces $S\subset \PPE$ of Picard rank $3$ containing the curve $C$.
In this case the bundle $\cE$ is generically of the form $\cE=\OO_{\PP^1}(1)\oplus \OO_{\PP^1}(1)\oplus \OO_{\PP^1}(1)\oplus \OO_{\PP^1}(1)\oplus \OO_{\PP^1}$. Hence the corresponding scroll $X$ has a zero-dimensional vertex. But since the curve generically will not pass through the vertex, we will consider $\PPE$ instead of $X$.

\subsection{Lattice polarized K3 surfaces} \label{section_lattice_pol_K3}
We recall the definition of the moduli space of lattice polarized $K3$ surfaces due to \cite{Do96}. 

For an even non-degenerate lattice  $\mathfrak{n}$ of signature $(1,r)$, an $\mathfrak{n}$-polarized $K3$ surface is a pair $(S,\varphi)$, where $S$ is a $K3$ surface and $\varphi:\mathfrak{n} \to \text{Pic}(S)$ is a primitive lattice embedding such that $\varphi(\mathfrak{n})$ intersects the ample cone $\text{Amp}(S)$. 
Two $\mathfrak{n}$-polarized $K3$ surfaces $(S,\varphi)$ and $(S',\varphi')$ are called isomorphic if there exists an isomorphism $\alpha: S\to S'$, such that 
$\varphi=\alpha^*\circ \varphi'$. 

In \cite{Do96} Dolgachev shows that the
moduli space $\cF^\mathfrak{n}$ parametrizing isomorphism classes of $\mathfrak{n}$-polarized $K3$ surfaces exists as an equidimensional quasi-projective variety of dimension $19-r$. $\cF^\mathfrak{n}$ has at most two components which are, by complex conjugation, interchanged on the period domain.

For the rest of this article we will denote by $\hh$ the rank $3$ intersection lattice with respect to the ordered basis $\{h_1,h_2,h_3 \}$
$$
\hh=\begin{pmatrix}14 & 16 & 5 \\ 16 & 16 & 6 \\ 5 & 6 & 0 \end{pmatrix}
$$
and consider the moduli space $\cF^\hh$.
If $(S,\varphi)\in \cF^\hh$ then we denote
$$
\OO_S(H)=\varphi(h_1), \ \OO_S(C)=\varphi(h_2) \text{ and } \OO_S(N)=\varphi(h_3).
$$
By abuse of notation we will also say that $\{\OsH,\OsC,\OsN\}$ forms a basis of $\hh$.

After suitable sign changes and Picard-Lefschetz reflections we may assume that $\OsH$ is big and nef (see \cite[VIII, Prop 3.10]{BHPV}). 
To check the ampleness of a class, it is sufficient to compute the intersection with all smooth rational curves, that is, curves with self-intersection $-2$ (see \cite[Ch. 2 Prop. 1.4]{Huy}).
A \emph{Maple} computation (see \cite{BH2}) shows that there are in fact many smooth rational curves on $S$ and  if $(S,\varphi)\in \cF^\hh$ such that $\varphi(\hh)=\Pic(S)$, then $\OsH$ intersects all of them positively. Hence $\OsH$ is ample.

We summarize several properties of the other relevant classes
All the statements in the following remark follow from classical results in \cite{SD} (see also \cite[Ch. 2]{Huy}) and lattice computations which are done in \cite{BH2}. 

\begin{remark}\label{remark_nef_bpf}
	We may assume that all basis elements of the lattice $\hh$ are effective. For a $K3$ surface $S\in \cF^\hh$ with $\Pic(S)=\hh$, such that $\OsH$ is ample, one can check that 
	\begin{itemize}
		\item $\OsH$ and $\OO_S(H-N)$ are ample, base point free and the generic elements in the linear systems are smooth.
		\item $\OsC$ is big and nef, base point free and the generic element in the linear system is smooth.
		\item $\OsN$ is nef and base point free and can  be represented by a smooth and irreducible elliptic curve.
	\end{itemize}
Although the assumption $\Pic(S)=\hh$ is only satisfied for very general $K3$ surfaces in $\cF^\hh$,  all conditions above are open in the moduli space.
We remark furthermore that for the lattice $\hh$ it can be checked that the ample class $\OO_S(H)$ determines the classes $\OO_S(C)$ and $\OO_S(N)$ (with desired intersection numbers) uniquely.
\end{remark}

As in \cite{Beau}, we denote $\cF^\hh_8$ the moduli space
$$
\cF^\hh_8=\big\{(S,\varphi)\ \big | \ (S,\varphi)\in \cF^\hh \ \text{and } \OsH \ \text{ample} \big\}.
$$
By the remark above $\cF^\hh_8$ is a Zariski open subset of $\cF^\hh$. 
In particular, $\cF^\hh_8$ is again a quasi-projective irreducible variety of dimension $17$. Moreover, $\cF^\hh_8$ is irreducible by \cite[Prop 5.9]{Do96} for this particular lattice.
In what follows, we will omit referring to the primitive lattice embedding $\varphi:\hh\to \Pic(S)$ for elements in $(S,\varphi)\in \cF^\hh_8$ most of the time.
Whenever we will consider the projective model $S\subset \PP^8$ of a $K3$ surface $S\in \cF^\hh_8$ we  identify $S$ with its image in $\PP(H^0(S,\OsH)^*)$.

Since for generic $S\in \cF^\hh_8$ the general element in the linear system $|\OsC|$ is a smooth curve of genus $9$, we may consider the open subset of the tautological $\PP^{9}$-bundle over the moduli space $\cF^\hh_8$
$$
\cP^\hh_8=\big\{(S,C)\ |\ S\in \cF_8^\hh\ \text{and }C\in |\OsC|\ \text{smooth}\big \}.
$$
In the next section we prove that $\cP^\hh_8$ is a $\PP^1$-bundle over the universal Brill--Noether variety $\cW^1_{9,6}$.

\section{\texorpdfstring{The space $\cP^\hh_8$ as a $\PP^1$-bundle over $\cW^1_{9,6}$}{}}\label{unirationality}
In this section we prove the dominance of the morphism
$$
\phi:\cP^\hh_8\to \cW^1_{9,6},\ (S,C)\mapsto (C,\OsN\otimes \OO_C)
$$
and conclude that $\cP^\hh_8$ is a $\PP^1$-bundle over an open subset of $\cW^1_{9,6}$.  

Some of the statements rely on a computational verification using \emph{Macaulay2} \cite{M2}. The \emph{Macaulay2} script which verifies all these statements can be found in \cite{BH3}.
We start over by showing that there exist $K3$ surfaces with the desired properties.
\begin{proposition}\label{existenceExample}
\hspace{0.1cm}
\begin{enumerate}
\item[(a)]
There exists a smooth canonical genus $9$ curve $C$ together with a line bundle  $L\in W^1_6(C)$ such that the relative canonical resolution has the form 
\end{enumerate}
\begin{equation*}
\resizebox{\textwidth}{!}{
$\mscr{I}_{C/\PPE} \leftarrow 
\begin{matrix} \OO_\PPE(-2H+R)^{\oplus 6}\\ \oplus \\ \OO_\PPE(-2H)^{\oplus 3} \\ \\ \\ \end{matrix}
\leftarrow 
\begin{matrix}\OO_\PPE(-3H+2R)^{\oplus 2}\\ \oplus\\ \OO_\PPE(-3H+R)^{\oplus 12}\\ \oplus \\\OO_\PPE(-3H)^{\oplus 2}\end{matrix}
\leftarrow
\begin{matrix}\\ \\ \OO_\PPE(-4H+2R)^{\oplus 3}\\ \oplus \\ \OO_\PPE(-4H+R)^{\oplus 6}\end{matrix}
\leftarrow \OO_\PPE(-6H+2R)\leftarrow 0.$}
\end{equation*}
\begin{enumerate}
\item[(b)] 
There exists a syzygy  $s:\OO_\PPE(-3H+2R)\to \OO_\PPE(-2H+R)^{\oplus 6}$ whose syzygy scheme defines a $K3$ surface $S \in \cF^\hh_8$. In particular, the general elements in the linear series $|\OsH |,| \OsC |,| \OsN |$ are smooth, irreducible and Clifford general. 
\end{enumerate}
\end{proposition}

\begin{proof}
Using \emph{Macaulay2}, we have implemented the construction of such curves together with the relative canonical resolution in  \cite{BH3}.
In our example the relative canonical resolution is of the form as stated in (a).

 A syzygy $s: \OO_\PPE(-3H+2R)\to \OO_\PPE(-2H+R)^{\oplus 6}$ is a generalized column 
 of the $6\times 2$ submatrix of the relative canonical resolution of $C\subset \PPE$. 
 The entries of $s$ span the four-dimensional vector space $H^0(\PPE,\OO_\PPE(H-R))$. 
 Let $f_1,\dots, f_6$ be the generators of $\cI_{C/\PPE}$ corresponding to $\OO_\PPE(-2H+R)^{\oplus 6}$. By definition of $s$ we have $(f_1,\dots, f_6)\cdot s=0$. After a base change we may assume that $s$ is of the form 
 $$
 s=(s_1,s_2,s_3,s_4,0,0)^t.
 $$
 Applying this base change to $f_1,\dots, f_6$, we get new generators $f_1',\dots, f_6'$ such that 
 $$(f_1',\dots, f_6')\cdot(s_1,s_2,s_3,s_4,0,0)^t=0.$$ In this case the syzygy scheme associated to $s$ is given by $\cI_{syz(s)}=\langle f_1',\dots, f_4'\rangle$. For the general definition of a syzygy scheme see \cite{vB}. 
 
 Again by \cite{BH3}, the image of the syzygy scheme in the scroll $X$, swept out by $|L|$, is the union of its vertex and a $K3$ surface $S\subset X\subset \PP^8$. Hence, after saturating with the vertex, we obtain a $K3$ surface $S\subset \PP^8$ of degree $14$ such that the ruling on $X$ defines an elliptic curve $N$ on $S$ and the hyperplane section $H$ is a canonical curve of genus $8$. The intersection products of the classes $\{\OsH,\OsC,\OsN\}$ define the lattice $\hh$. 
\end{proof}

\begin{lemma}\label{K3_on_scroll}
Let $(S, C)\in \cP^\hh_8$ be general. 
Then $L=\OsN\otimes \OO_C$ defines a $g^1_6$ on $C$ such that $S$ is contained in the scroll  $X=\bigcup_{D\in |L|}\overline{D}$ swept out by $|L|$.
\end{lemma}

\begin{proof}
 Let $H\in |\OsH|$ be a general element and let $N\in |\OsN|$ be an elliptic curve of degree $5$. 
 Assume that the span $\overline{N}\cong \PP^3$ is three-dimensional. Then the intersection $N\cap H$ consists of $5$ points and the span $\overline{N\cap H}$ is a $\PP^2$. But this would give a $g^2_5$ by the geometric version of Riemann-Roch. Because of the genus formula we have $W^2_{5}(H)=\emptyset$, a contradiction. Thus, $N$ is an elliptic normal curve and $\overline{N}\cong \PP^4$. 
 
 Now since $S\subset \bigcup_{N\in |\OsN|}\overline{N}$  it remains to show that $\overline{N\cap C}\cong \PP^4$. 
 The intersection $N\cap C$ consists of $6$ points. Assume that these $6$ points only span a hyperplane $h\cong \PP^3\subset \PP^4$. Then $\deg(h\cap N)>\deg(N)$ which means, by B\'ezout, that $h\cap N$ is a component of $N$. Thus, the general $N$ is reducible, a contradiction by Remark \ref{remark_nef_bpf}. 
\end{proof}

\begin{lemma}\label{shape_res_K3}
	Let $(S, C)\in \cP^\hh_8$ be general
	 and $L=\OO_S(N)\otimes \OO_C$ 
	such that the relative canonical resolution of $C\subset \PPE$ has a balanced first syzygy bundle. If we further assume that $S\subset \PPE$, where $\PPE$ is the scroll associated to $L$, then $S\subset \PPE$ has a resolution of the form 
	$$
	0\leftarrow \OO_{S/\PPE} \leftarrow 
	\begin{matrix} \OO_\PPE(-2H+R)^{\oplus 4}\\ \oplus \\ \OO_\PPE(-2H)^{\oplus 1} \end{matrix}
\overset{\Psi}{	\longleftarrow }
	\begin{matrix}\OO_\PPE(-3H+2R)^{\oplus 1}\\ \oplus\\ \OO_\PPE(-3H+R)^{\oplus 4}\end{matrix}
	\leftarrow
	\OO_\PPE(-5H+2R)
	\leftarrow 0
	$$
	for a skewsymmetric matrix $\Psi$ and is generated by the $5$ Pfaffians of the matrix $\Psi$.
\end{lemma}

\begin{proof}
	The surface $S\subset \PPE$ is Gorenstein of codimension $3$, and therefore it follows by the Buchsbaum-Eisenbud structure theorem \cite{BE77} that $S$ is generated by the Pfaffians of a skew-symmetric matrix $\Psi$ and has (up to twist) a self-dual resolution. 
	The shape of the resolution of $S\subset \PPE$ is the same as the shape of the resolution of $S\cap H\subset \PPE\cap H$ for a general hyperplane $H$. 
	Since we assume $(S,C)\in \cP^\hh_8$ to be general, $S\cap H$ is a $5$-gonal genus $8$ curve (as in  Proposition \ref{existenceExample}) and $\PPE\cap H$ is a $4$ dimensional variety of degree $4$, hence isomorphic to a scroll 
	 $\PP(\cE')$.
	  By \cite{Sch} we know that $S\cap H\subset \PP(\cE')$ is generated by the $5$ Pfaffians of a skew-symmetric $5\times 5$ matrix and therefore also $\Psi$ needs to be a $5\times 5$ matrix. It remains to determine the balancing type.
	 By our assumption $C\subset \PPE$ has a balanced first syzygy bundle as in Proposition \ref{existenceExample}. Since the relative linear strand of the resolution of $S\subset \PPE$ is a subcomplex  of the relative linear strand in the resolution of $C\subset \PPE$, we obtain that the resolution of $S\subset \PPE$ has the following form
	 	$$
	 	0\leftarrow \cI_{S/\PPE} \leftarrow 
	 	\begin{matrix} \OO_\PPE(-2H+R)^{\oplus a_1}\\ \oplus \\ \OO_\PPE(-2H)^{\oplus a_2} \end{matrix}
	 	\overset{\Psi}{	\longleftarrow }
	 	\begin{matrix}\OO_\PPE(-3H+2R)^{\oplus b_2}\\ \oplus\\ \OO_\PPE(-3H+R)^{\oplus b_1}\end{matrix}
	 	\leftarrow
	 	\OO_\PPE(-5H+2R)
	 	\leftarrow 0
	 	$$
	 with $a_i=b_i$ for $i=1,2$ and $a_1+a_2=5$. 
	 By taking the first Chern classes of the bundles appearing in the resolution above we get
	 $$ \big(b_2\cdot(-3H+ 2R)+b_1\cdot (-3H+ R)\big)- \big(a_1\cdot (-2H+R)+a_2\cdot(-2H)\big)=-5H+2R
	$$
	and hence, $2b_2+b_1-a_1=2$. Therefore, $b_2=1$ and the only possible shape for the resolution $S\subset \PPE$ is the one in the lemma.
\end{proof}

\begin{corollary}\label{K3_forces_unbalancedness}
Let $(S, C)\in \cP^\hh_8$ be a general element and let $L=\OO_S(N)\otimes \OO_C$. Then the relative canonical resolution of $C\subset \PPE$ has an unbalanced second syzygy bundle where $\PPE$ is the scroll associated to $L$.	
\end{corollary}

\begin{proof}
For a general pair $(S,C)\in \cP_8^\hh$ the class $\OO_S(N)$ is nef. Thus by Lemma \ref{K3_on_scroll} it follows that $S$ is contained in the scroll $\PPE$ defined by $L=\OO_S(N)\otimes \OO_C$. Note that having a balanced first syzygy bundle in the relative canonical resolution is an open condition. Therefore, by Proposition \ref{existenceExample}  $C\subset \PPE$ has a balanced first syzygy bundle and we can apply the previous lemma.

 Since the relative linear strand of $S\subset \PPE$ is a subcomplex of the relative linear strand of the  resolution of $C\subset \PPE$, it follows from Lemma \ref{shape_res_K3} that the resolution of $C\subset \PPE$ has an unbalanced second syzygy bundle.
\end{proof}


By the above corollary it follows for $(S,C)\in \cP^\hh_8$ general that $C\subset \PPE$ has a second syzygy bundle of the form 
$$
\OO_\PPE(-3H+2R)^{\oplus a} \oplus \OO_\PPE(-3H+R)^{\oplus (16-2a)} \oplus \OO_\PPE(-3H)^{\oplus a},
$$
for some $a\geq 1$. The next lemma relates the balancing type of the second syzygy bundle to the fiberdimension of the morphism $\phi: \cP_8^\hh \to \cW^1_{9,6}$.

\begin{lemma}\label{bound_fiber_dimension}
Let $(S, C)\in \cP^\hh_8$ and $L=\OO_S(N)\otimes \OO_C$ such that the relative resolution of $S\in \PPE$ is of the form as in Lemma \ref{shape_res_K3}. Then 
 the $K3$ surface $S$ is uniquely determined by subcomplex
\begin{align*}
0\leftarrow \OO_{S/\PPE} \leftarrow 
\OO_\PPE(-2H+R)^{\oplus 4}
\leftarrow 
\OO_\PPE(-3H+2R)^{\oplus 1}
\end{align*}
of the relative canonical resolution of $C\subset \PPE$. 
In particular, the fiber dimension of $\phi$ is bounded by $a-1$ where $a$ is the rank of the subbundle $\OO_\PPE(-3H+2R)^{\oplus a}$ in the relative canonical resolution of $C\subset \PPE$. 
\end{lemma}

\begin{proof}
Let $q_1,\dots, q_4\in H^0(\PPE,\OO_\PPE(2H-R))$ be the entries of the matrix $\OO_{S/\PPE} \leftarrow \OO_\PPE(-2H+R)^{\oplus 4}$ and $l_1,\dots,l_4\in H^0(\PPE,\OO_\PPE(H-R))$ be the entries of $\OO_\PPE(-2H+R)^{\oplus 4}
\leftarrow 
\OO_\PPE(-3H+2R)^{\oplus 1}$. Then by \cite[Lemma 4.2]{Sch91} there exists a skew-symmetric $4\times 4$ matrix $A=(a_{i,j})_{i,j=1,\dots,4}$ such that 
$$
q_i =\sum_{j=1}^4 a_{i,j}l_i.
$$
and the $5$th Pfaffian $q_5\in H^0(\PPE,\OO_\PPE(2H))$ defining the surface $S$ is given as $\text{Pf}(A)$. 
So $q_1,\dots q_5$ are the Pfaffians of the $5\times 5$ matrix
$$
\psi=\begin{pmatrix} 
0 & -l_1 & -l_2 & -l_3 & -l_4 \\
l_1 & 0 & -a_{3,4} & a_{2,4} & -a_{2,3} \\
l_2 & a_{3,4} & 0 & a_{1,4} & -a_{1,3} \\
l_3 & -a_{2,4} & -a_{1,4}& 0 & a_{1,2} \\
l_4 & a_{2,3} & a_{1,3} & -a_{1,2} & 0
\end{pmatrix}
$$
Considering the Koszul resolution associated to the section $(l_1,\dots,l_4)\in H^0(\OO_\PPE(H-R))^4$ we get
	$$
	\bigwedge ^3\OO(-H+R)^4 \to \bigwedge ^2\OO(-H+R)^4 \to \OO(-H+R)^4 \to \OO
	$$
with $\bigwedge^2\OO(-H+R)^4=\OO(-2H+2R)^6$	and $\bigwedge^3\OO(-H+R)^4=\OO(-3H+3R)^4$. So tensoring the whole sequence with $\OO(3H-2R)$
 we get 
\begin{small}
$$
\xymatrix@C=1.4em{
	& \OO_\PPE\ar[dl] \ar[d]^{\exists !} \ar[dr]^{(q_1,\dots,q_4)^t} \\
	\bigwedge^3 \OO_\PPE^4 \otimes \OO_\PPE(R) \ar[r]^{\varphi}        &  \bigwedge ^2 \OO_\PPE^4 \otimes \OO_\PPE(H) \ar[r]   &
	\OO_\PPE(2H-R)^4 \ar[rr]^{(l_1,\dots,l_4)} &&\OO_\PPE(3H-2R) \ar[r] &0   
	}
$$
\end{small}
The space $H^0(\PPE, \bigwedge ^2 \OO_\PPE^4 \otimes \OO_\PPE(H))$ parametrizes skew-symmetric $4\times 4$ matrices with entries in $H^0(\PPE, \OO_\PPE(H))$. Fixing the $4$ Pfaffians $q_1,\dots,q_4$ together with their syzygy $(l_1,\dots,l_4)$ we see that the matrix $A$ and hence the matrix $\Psi$ is unique up to the image of $\varphi$. We identify an element $e_i\wedge e_j\in H^0(\PPE, \bigwedge ^2 \OO_\PPE^4 \otimes \OO_\PPE(H)) $ with the  skew-symmetric matrix 
where the index of the only nonzero entries is precisely $\{k,l\}=\{1,\dots,4\}\setminus \{i,j\}$. The image of $\varphi$ consists of those matrices which are obtained by the operation of the first column (resp. row) of $\Psi$ on $A$ which respects the skew-symmetric structure.
\end{proof}

Note again that $\dim \cP^\hh_8= \dim \cW^1_{9,6}+1$. To show the following theorem we show that the general fiber is at most one-dimensional.

\begin{theorem}\label{mainthm}
The morphism
$$
\phi: \cP_8^\hh\to \cW^1_{9,6},\ (S,C)\mapsto \big(C, \OsN\otimes \OO_C\big)
$$
is dominant. 
\end{theorem}

\begin{proof}
The morphism $\phi: \cP^\hh_8\to \cW^1_{9,6}$ is locally of finite type since $\cP^\hh_8$ and $\cW^1_{9,6}$ are algebraic quasi-projective varieties (and hence schemes of finite type over $\Bbbk$). Therefore by Chevalley's Theorem \cite[Thm. 13.1.3]{EGAIV} the map 
$\cP_8^\hh \to \mathbb{Z},\ x\mapsto \dim_x \phi^{-1}(\phi(x))$ is upper semicontinuous. \newline
By Proposition \ref{existenceExample} we obtain
a point $(C,L)\in \cW^1_{9,6}$ in the image of $\phi$.
The preimage in part (b) of Proposition \ref{existenceExample}, constructed via syzygy schemes, satisfies all generality assumptions in the previous lemmata.
 Now Lemma \ref{K3_on_scroll} implies that a general $K3$ surface in the fiber over $(C,L)$ is contained in the $5$-dimensional scroll $\PPE$, defined by the pencil $L$ on $C$. \newline
 By Corollary \ref{K3_forces_unbalancedness} it follows that such K3 surfaces $S\subset \PPE$ in the fiber are defined by the Pfaffians of a skew symmetric $5\times 5$ matrix 
	$$
	\begin{matrix} \OO_\PPE(-2H+R)^{\oplus 4}\\ \oplus \\ \OO_\PPE(-2H)^{\oplus 1} \end{matrix}
	\overset{\Psi}{	\longleftarrow }
	\begin{matrix}\OO_\PPE(-3H+2R)^{\oplus 1}\\ \oplus\\ \OO_\PPE(-3H+R)^{\oplus 4} .\end{matrix}
	$$
 Since the relative linear strand of $S\subset \PPE$ is a subcomplex of the relative linear strand of $C\in \PPE$, it follows from the shape of resolution and Lemma \ref{bound_fiber_dimension} that the fiber over $(C,L)$ is at most $1$-dimensional. By semicontinuity it follows that $\dim_x\phi^{-1}(\phi(x))\le1$ for all $x$ in some open subset  $U\subset \cP^\hh_8$. 
Now since $\phi$ is a morphism of algebraic quasi-projective varieties, 
we have a dominant map $\phi: \cP_8^\hh\to \overline{\im(\phi)}$. 
The space $\cP^\hh_8$ is equidimensional and we get
$$
\dim\cP^\hh_8=\dim \overline{\im(\phi)} + \dim_x\phi^{-1}(\phi(x))\le \dim  \cW^1_{9,6} + 1.
$$
Since $\dim \cP_8^\hh=26$ and $\dim  \cW^1_{9,6}=25$, we obtain $\dim \overline{\im(\phi)}=\dim \cW^1_{9,6}$. The universal Brill--Noether variety $\cW^1_{9,6}$ is irreducible and therefore
it follows that the image of $\phi$ and hence $\phi(U)$ is also dense in $\cW^1_{9,6}$.
\end{proof}

\begin{corollary}\label{unirationalityTheorem}
The general fiber of $\phi$ is a rational curve parametrized by syzygy schemes as in part (b) of Proposition \ref{existenceExample}.	
The moduli space $\cP_8^\hh$  is birational to a $\PP^1$-bundle over an open subset of $ \cW^1_{9,6}$. In particular $\cP_8^\hh$, $\cF_8^\hh$ and hence $\cF^\hh$ are unirational.
\end{corollary}

\begin{proof}
By Theorem \ref{mainthm} the map $\phi:\cP^\hh_8\to \cW^1_{9,6}$ is dominant. 
 Thus by Proposition \ref{existenceExample} the general element in $\cW^1_{9,6}$ has a relative canonical resolution with second syzygy bundle of the form
$$
\OO_\PPE(-3H+2R)^{\oplus 2} \oplus \OO_\PPE(-3H+R)^{\oplus 12} \oplus \OO_\PPE(-3H)^{\oplus 2}
$$
and therefore, by the dominance of $\phi$ and Lemma \ref{bound_fiber_dimension},
 the construction in Proposition \ref{existenceExample} holds in an open set. To be more precise, over an open subset of
 of $ \cW^1_{9,6}$ the syzygy schemes defined by syzygies in the free $\OO_\PPE$-module $\OO_\PPE(-3H+2R)^{\oplus 2}$ 
correspond 
 (after saturating with the vertex of the scroll) to the $K3$ surfaces in the fiber of $\phi$.
Therefore we obtain a birational map
$$
\widetilde{\phi}: \cP^\hh_8\to \widetilde{\cW^1_{9,6}}
$$ 
where 
$$
\widetilde{\cW^1_{9,6}}=
\bigg\{(C,L,s)\ \big| \ (C,L)\in \cW^1_{9,6},\ s\in \OO_\PPE(-3H+2R)^{\oplus 2}\bigg \}
$$
is a  $\PP^1$-bundle over an open dense subset of $\cW^1_{9,6}$. 
Now  $\cW^1_{9,6}$ is classically known to be unirational (see \cite{seg} and \cite{AC81}) and hence $\cP^\hh_8$ is unirational as well.
\end{proof}

\begin{corollary}\label{allUnbalanced}
 For any $(C,L)\in \cW^1_{9,6}$, the relative canonical resolution has an unbalanced second syzygy bundle. 
\end{corollary}

\begin{proof}
 Having a balanced second syzygy bundle is an open condition in $\cW^1_{9,6}$ (by semicontinuity of $h^1(\PP^1,\cEnd(N_2))$). The claim follows from the fact that the general point in $\cW^1_{9,6}$ has an unbalanced second syzygy bundle by Corollary \ref{unirationalityTheorem}. 
\end{proof}

\begin{remark}	
	There exists a unirational codimension $4$ subvariety $V\subset \cW^1_{9,6}$, parametrizing pairs $(C,L)$ such that $C$ is the rank one locus of a certain $3\times 3$ matrix defined on the scroll $\PPE$ swept out by $|L|$ (see \cite[Section 4.3]{Gei13}). 
	
	Although there is in general no structure theorem for resolutions of Gorenstein subschemes of codimension $\geq 4$, the relative canonical resolution of elements parametrized by $V$ is given by a  so-called Gulliksen--Negard complex. 
	
	One can check that the splitting types of the bundles in the Gulliksen--Negard complex are the same as in Proposition \ref{existenceExample}. 
	However, the subvariety $V$ does not lie in the image of the map $\phi:\cP^\hh_8 \to \cW^1_{9,6}$. 
	Indeed, since curves parametrized by $V$ are degeneracy loci of $3\times 3$ matrices, all linear syzygies (as in Proposition \ref{existenceExample} (b)) have rank $3$. Therefore, the corresponding syzygy schemes do not define $K3$ surfaces.
\end{remark}

\begin{remark}
	In \cite{Muk02} Mukai showed that a transversal linear section $\PP^8\cap G(2,6)\subset \PP^{14}$ of the embedded Grassmannian $G(2,6)\subset \PP^{14}$ is a Brill--Noether general $K3$ surface and that every Brill--Noether general $K3$ surface arises in this way. 
	One can show that a very general surface $S\in \cF^\hh_8$ is indeed Brill--Noether general and therefore arises as a transversal linear section of $G(2,6)$. 
	Among other things we will show in a forthcoming work that also the generators of the Picard group $\Pic(S)$ are obtained by taking linear sections of subvarieties inside $G(2,6)$. 
	To be more precise, changing the basis of the lattice $\hh$ to 
	$\{\OsH, \OO_S(Q)=\OO_S(C-H), \OsN \}$, we have $Q=\PP^8 \cap G(2,4)$ and  $N=\PP^8 \cap G(2,5)$ for Grassmannians $G(2,4),  G(2,5)\subset G(2,6)\subset \PP^{14}$ not containing each other. 
\end{remark}

\section{\texorpdfstring{A birational description of $\cW^3_{9,10}$}{}}\label{birationalityBNloci}

The Serre dual of a $g^1_6$ on a general genus $9$ curve $C$ is a $g^3_{10}$ defining an embedding into $\PP^3$.
Let $C'$ be the image of a general genus $9$ curve $C$ under the residual map 
$$
C\overset{|\omega\otimes L^{-1}|}{\longrightarrow} C'\subset \PP^3.
$$
Then all maps in the long exact cohomology sequence induced by the sequence
$$
0\to \cI_{C'/\PP^3}(n)\to \OO_{\PP^3}(n)\to \OO_{C'}(n) \to 0
$$
have maximal rank and $C'$ is contained in a net of quartics whose general element is smooth (see \cite{BH3}). 
Let $\nn$ be the rank $r\geq 2$  Picard lattice of a  very general quartic in this family. We fix a basis $\{n_1,n_2,\dots \}$ for $\nn$ with $n_1^2=4, n_2^2=16$ and $n_1n_2=10$ and consider the moduli space 
$$\cF^\nn_3=\big\{(S,\varphi)\ \big| \ (S,\varphi)\in \cF^\nn \text{ and}\ \OO_S(H')=\varphi(n_1) \text{ ample }\big \}$$
and the open subset of the tautological bundle
$$
\cP^\nn_3=\big\{(S,\varphi,C)\ \big| \ (S,\varphi)\in \cF^\nn_3 \text{ and}\ C\in |\varphi(n_2)| \text{ smooth }\big \}.
$$
We get a dominant map 
$$\cP_3^\nn\to \cW^3_{9,10}\cong \cW^1_{9,6},\ (S,\varphi,C)\mapsto (C,\OO_S(H')\otimes \OO_{C})$$
whose general fiber has dimension $2$.
Now, since 
$$
\dim \cP^\nn_3= \dim \cF_3^\mathfrak{n}+\dim |C'| =(20-r)+9 = \dim \cW^1_{9,6}+2=27,
$$
we see that $\nn$ is a rank $2$ lattice and hence the Picard lattice of a very general
$K3$ surface in $\cF^\nn_3$ is generated by the class of a plane quartic and the class of $C$.
As a consequence of Theorem \ref{mainthm} and Corollary \ref{unirationalityTheorem} we now obtain:

\begin{corollary}
With notation as above $\cP^\nn_3 \to \cW^3_{9,10}$ is a $\PP^2$-bundle over an open subset of $\cW^3_{9,10}$. The general fiber contains a rational curve parametrizing $K3$ surfaces contained in $\cF_8^\hh$.
\end{corollary}

\begin{proposition}\label{existenceExample2}
	With notation as above, there exists a pair $(C,\omega_C\otimes L^{-1})\in \cW^3_{9,10}$ such that \begin{enumerate}
		\item[(1)] $V=H^0(\PP^3, \cI_{C'}(4))$ is $3$-dimensional,
		\item[(2)] the plane nodal rational curve $\Gamma$ in $\PP(V)$, whose points correspond to $K3$ surfaces given as syzygy schemes as in Proposition \ref{existenceExample}, has degree $3$ and
		\item[(3)] the abstract $K3$ surface $S_p$ corresponding to the unique singular point $p$ of the rational curve $\Gamma$ has a smooth model in $\PP^3$.
	\end{enumerate}
\end{proposition}

\begin{proof}
	We verify the above statement in our \emph{Macaulay2}-script \cite{BH3}.
\end{proof}

In the following we describe the Picard lattice of $S_p$.
Recall that the linear syzygy in the relative canonical resolution of a surface $S$ in $\cF^\hh_8$ determines the polarized $K3$ surface $(S,\OO_S(H))$ uniquely by Proposition \ref{bound_fiber_dimension}. 
Hence, all $K3$ surfaces $(S,\OO_S(H))$ given as syzygy schemes as in Proposition \ref{existenceExample} are non-isomorphic (as polarized $K3$ surfaces). Therefore, to be a nodal point of the rational curve $\Gamma$ means that there are two $K3$ surfaces in $\PP^8$ mapping to the same quartic in $\PP^3$.
In other words, the Picard group $\Pic(S_p)$ contains two (pseudo-) polarizations $\OO_{S_p}(H_1)$ and $\OO_{S_p}(H_2)$ (and corresponding elliptic classes $\OO_{S_p}(N_i)$, $i=1,2$) such that $H_i^2=14$, $H_i.C=16$ and $|\OO_{S_p}(H_1-N_1)|=|\OO_{S_p}(H_2-N_2)|$. 
Note that the image of $S_p$ in $\PP^3$ is given by $|\OO_{S_p}(H_i-N_i)|$.

Since by Section \ref{section_lattice_pol_K3} fixing two basis elements $C$ and $N$ (or equivalently $(H-N)$) for the lattice $\hh$ determines the third class $H$ uniquely, it follows that $\Pic(S_p)$ has rank at least $4$. 

Thus, $\Pic(S_p)$ contains a lattice of the following form 
$$
\begin{pmatrix}
	14 & 16 & 5 & a\\
	16 & 16 & 6 &16 \\
	5 & 6 & 0 &b \\
	a &16 &b & 14
\end{pmatrix}
$$
with respect to an ordered basis $\big\{\OO_{S_p}(H_1), \OO_{S_p}(C), \OO_{S_p}(N_1), \OO_{S_p}(H_2) \big \}$ and $a,b$ integers.
Using that $(C-H_i)$ is a $(-2)$-curve for $i=1,2$, an easy computation yields $a=16$ and $b=6$.
We remark that $H_1.(C-H_2)=0$ and therefore $\OO_{S_p}(H_i)$ does not define an ample class on $S_p$. Hence, the surface $S_p$ lies in the boundary of $\cF^\hh_8$.
Nevertheless, the definition of the moduli space $\cF^\hh$ can be extended to pseudo polarized $K3$ surfaces (see \cite{Do96}).
If we change the basis of the above lattice to 
$$
\big\{\OO_{S_p}(H')=\OO_{S_p}(H_i-N_i), \OO_{S_p}(C), \OO_{S_p}(Q_1)=\OO_{S_p}(C-H_1), \OO_{S_p}(Q_2)=\OO_{S_p}(C-H_2)  \big\}
$$
then the corresponding intersection matrix has the following form
$$
\hh'\sim 
\begin{pmatrix}
4 & 10 & 1 & 1 \\
10 & 16 & 0 & 0 \\
1 & 0 & -2 & 0 \\
1 & 0 & 0 & -2
\end{pmatrix}
$$

We denote by $\hh'$ be the (abstract) rank $4$ lattice which is defined by the intersection matrix above with respect to some fixed basis $\{h_1',h_2',h_3',h_4'\}$. For a lattice polarized $K3$ surface $(S,\varphi)\in \cF^{\hh'}$ we denote
$$
\OO_S(H')=\varphi(h_1'),\ \OO_S(C)=\varphi(h_2'),\  \OO_S(Q_1)=\varphi(h_3') \text{ and } \OO_S(Q_2)=\varphi(h_4').
$$
Again, we will omit referring to the primitive lattice embedding $\varphi:\hh'\to \Pic(S)$ for elements $(S,\varphi)\in \cF^{\hh'}$ and we will say that $\{ \OO_{S}(H'), \OO_{S}(C), \OO_{S}(Q_1), \OO_{S}(Q_2) \}$ forms a basis of $\hh'$.

As for the lattice $\hh$ one can check using \emph{Maple} (see \cite{BH2}) that for a surface 
$S\in \cF^{\hh'}$ with $\Pic(S)=\hh'$ the class $\OO_S(H')$ is ample. We consider again the open subset
$$
\cF^{\hh'}_3:=\big\{S\ \big |\ S\in \cF^{\hh'} \text{ and } \OO_S(H') \text{ ample}\big\}
$$
of the moduli space $\cF^{\hh'}$ and the open subset of the tautological $\PP^9$-bundle over $\cF^{\hh'}_8$
$$
\cP^{\hh'}_3=\big\{ (S, C)\ \big | \ S\in \cF^{\hh'}_3 \ \text{and } C\in |\OO_S(C)| \ \text{smooth}    \big \}.
$$ 
Furthermore, the class $\OO_S(H')$ determines the classes $\OO_S(C),\OO_S(Q_1)$ and $\OO_S(Q_2)$ (with desired intersection numbers) uniquely. Hence, we get generic injections 
$$\cF^{\hh'}_3\hookrightarrow \cF^\hh_3 \hookrightarrow \cF^\nn_3 \hookrightarrow \cF_3$$ 
into the moduli space of polarized $K3$ surfaces of genus $3$.

\begin{theorem}\label{birationalityTheorem}
	The morphism 
	$$
	\phi': \cP^{\hh'}_3 \to \cW^3_{9,10},\ (S,C) \mapsto (C, \OO_S(H')\otimes \OO_C)
	$$
	defines a birational equivalence. In particular $ \cP^{\hh'}_3$, $ \cF^{\hh'}_3$ and $ \cF^{\hh'}$ are unirational.
\end{theorem}

\begin{proof}
We proceed as in the proof of Theorem \ref{mainthm}.
 By Proposition \ref{existenceExample2} and the preceeding discussion there exists a pair $(C,\omega_C\otimes L^{-1})\in \cW^3_{9,10}$ in the image of the map $\phi'$. Furthermore, every point in the fiber corresponds to a singular point of rational curve $\Gamma$ as in Proposition \ref{existenceExample2}. 
 Indeed, the spaces $\cF^\hh_8$ and 
 $\cF^\hh_3:=\{S\ |\ S\in \cF^{\hh} \text{ and } \OO_S(H')=\varphi(h_1-h_3) \text{ ample} \}$ 
 are birational (the mapping $\OO_S(H)\mapsto \OO_S(H-N)$ is defined on an open subset and for a very general $K3$ surface $S\in \cF^\hh$, it is equivalent to choose a polarization $\OO_S(H)$ or a polarization $\OO_S(H-N)$). 
 Therefore, by Theorem \ref{mainthm} and Corollary \ref{unirationalityTheorem} we get a dominant morphism $\widetilde{\phi}:\cP^\hh_3\to \cW^3_{9,10}$ whose fibers are rational curves which we identify with $\Gamma$. Hence, the fiber of $\phi'$ is contained in the fiber of the map $\widetilde{\phi}$ and we get the following diagram  
 $$
 \begin{xy}
  \xymatrix{
  \Sing(\Gamma)\ \ar@{^(->}[r] \ar@{|->}[rdd]_{\phi'} & \Gamma\  \ar@{^(->}[r] \ar@{|->}[dd]_{\widetilde{\phi}} & \PP^2  \ar@{|->}[ldd] & & \cP^{\hh'}_3\  \ar@{^(-->}[r] \ar[rdd]_{\phi'} & \cP^\hh_3\  \ar@{^(-->}[r] \ar[dd]_{\widetilde{\phi}} & \cP^\nn_3 \ar[ldd]^{\PP^2-\text{bundle}} \\  
  & &  \mathlarger{\ar@{^(-->}[rr]} & &  & & \\
  & (C,\omega_C\otimes L^{-1}) & & \in & & \ \ \ \ \cW^3_{10,9} \ \ \ \ & 
  }
 \end{xy}
 $$
 The dimension of $\cP^{\hh'}_3$ is 
 $$
 \dim \cP^{\hh'}_3 = 20 - \rk( \hh') + g(C) = 16 + 9 = 25 = \dim \cW^3_{9,10}
 $$
 and both spaces are irreducible.
 Thus, by upper-semicontinuity on the fiber dimension the map $\phi'$ is generically finite and dominant. It remains to show the generic injectivity.
 
 In the example of Proposition \ref{existenceExample2} the fiber of $\widetilde{\phi}$ is a rational plane cubic, and hence, has a unique singular point. In the last part of our \emph{Macaulay2}-file \cite{BH3} we verify that this is the generic behaviour: 
 A general pair $(C,\omega_C\otimes L^{-1})$ in the image of $\widetilde{\phi}$ gives rise to an unbalanced relative canonical resolution as in Proposition \ref{existenceExample}. The rational curve of $K3$ surfaces given as the fiber of  $\widetilde{\phi}$ corresponds to a one-dimensional family of generic syzygy schemes cut out by the maximal Pfaffians of $5\times 5$ skew-symmetric matrices. We show that the generic one-dimensional family of such matrices always induces a rational cubic (see also the explanation in \cite[Part 6]{BH3}). 
 We conclude that the dominant morphism $\phi': \cP^{\hh'}_3\to \cW^3_{9,10}$ is generically injective and therefore defines a birational equivalence.
\end{proof}


\end{document}